\documentclass[leqno,11pt]{amsart}

\usepackage{amssymb, amsmath}
\usepackage{amssymb,amsfonts}
\usepackage{amsmath,latexsym}
\usepackage{pdfsync}
\usepackage{bbm}

\newtheorem{theo}{Theorem}[section]
\newtheorem{prop}{Proposition}[section]
\newtheorem{lemme}{Lemma}[section]
\newtheorem{defin}{Definition}[section]

\begin{document}

%STYLE FRANCAIS

%SOMMAIRE
\makeatletter
\def\sommaire{\@restonecolfalse\if@twocolumn\@restonecoltrue\onecolumn
\fi\chapter*{Sommaire\@mkboth{SOMMAIRE}{SOMMAIRE}}
  \@starttoc{toc}\if@restonecol\twocolumn\fi}
\makeatother

%BIBLIOGRAPHIE
\makeatletter
\def\thebibliographie#1{\chapter*{Bibliographie\@mkboth
  {BIBLIOGRAPHIE}{BIBLIOGRAPHIE}}\list
  {[\arabic{enumi}]}{\settowidth\labelwidth{[#1]}\leftmargin\labelwidth
  \advance\leftmargin\labelsep
  \usecounter{enumi}}
  \def\newblock{\hskip .11em plus .33em minus .07em}
  \sloppy\clubpenalty4000\widowpenalty4000
  \sfcode`\.=1000\relax}
\let\endthebibliography=\endlist
\makeatother

%REFERENCES
\makeatletter
\def\references#1{\section*{R\'ef\'erences\@mkboth
  {R\'EF\'ERENCES}{R\'EF\'ERENCES}}\list
  {[\arabic{enumi}]}{\settowidth\labelwidth{[#1]}\leftmargin\labelwidth
  \advance\leftmargin\labelsep
  \usecounter{enumi}}
  \def\newblock{\hskip .11em plus .33em minus .07em}
  \sloppy\clubpenalty4000\widowpenalty4000
  \sfcode`\.=1000\relax}

\let\endthebibliography=\endlist
\makeatother

%MACROS AVEC ARGUMENTS%
\def\lunloc#1#2{L^1_{loc}(#1 ; #2)}
\def\bornva#1#2{L^\infty(#1 ; #2)}
\def\bornlocva#1#2{L_{loc}^\infty(#1 ;\penalty-100{#2})}
\def\integ#1#2#3#4{\int_{#1}^{#2}#3d#4}
\def\reel#1{\R^#1}
\def\norm#1#2{\|#1\|_{#2}}
\def\normsup#1{\|#1\|_{L^\infty}}
\def\normld#1{\|#1\|_{L^2}}
\def\nsob#1#2{|#1|_{#2}}
\def\normbornva#1#2#3{\|#1\|_{L^\infty({#2};{#3})}}
\def\refer#1{~\ref{#1}}
\def\refeq#1{~(\ref{#1})}
\def\ccite#1{~\cite{#1}}
\def\pagerefer#1{page~\pageref{#1}}
\def\referloin#1{~\ref{#1} page~\pageref{#1}}
\def\refeqloin#1{~(\ref{#1}) page~\pageref{#1}}
\def\suite#1#2#3{(#1_{#2})_{#2\in {#3}}}
\def\ssuite#1#2#3{\hbox{suite}\ (#1_{#2})_{#2\in {#3}}}
\def\longformule#1#2{
\displaylines{
\qquad{#1}
\hfill\cr
\hfill {#2}
\qquad\cr
}
}
\def\inte#1{
\displaystyle\mathop{#1\kern0pt}^\circ
}
\def\sumetage#1#2{\sum_{\substack{{#1}\\{#2}}}}
\def\limetage#1#2{\lim_{\substack{{#1}\\{#2}}}}
\def\infetage#1#2{\inf_{\substack{{#1}\\{#2}}}}
\def\maxetage#1#2{\max_{\substack{{#1}\\{#2}}}}
\def\supetage#1#2{\sup_{\substack{{#1}\\{#2}}}}
\def\prodetage#1#2{\prod_{\substack{{#1}\\{#2}}}}

\def\convm#1{\mathop{\star}\limits_{#1}
}
\def\vect#1{
\overrightarrow{#1}
}
\def\Hd#1{{\mathcal H}^{d/2+1}_{1,{#1}}}

%MACROS MECAFLU%
\def\derconv#1{\partial_t#1 + v\cdot\nabla #1}
\def\esptourb{\sigma + L^2(\R^2;\R^2)}
\def\tourb{tour\bil\-lon}

%ABREVIATIONS%

%\newcommand{\beq}{\begin{eqnarray}}
%\newcommand{\eeq}{\end{eqnarray}}
%\newcommand{\bq}{\begin{equation}}
%\newcommand{\eq}{\end{equation}}
%\newcommand{\beqn}{\begin{eqnarray*}}
%\newcommand{\eeqn}{\end{eqnarray*}}

\newcommand{\beq}{\begin{equation}}
\newcommand{\eeq}{\end{equation}}
\newcommand{\ben}{\begin{eqnarray}}
\newcommand{\een}{\end{eqnarray}}
\newcommand{\beno}{\begin{eqnarray*}}
\newcommand{\eeno}{\end{eqnarray*}}

\let\al=\alpha
\let\b=\beta
\let\g=\gamma
\let\d=\delta
\let\e=\varepsilon
\let\z=\zeta
\let\lam=\lambda
\let\r=\rho
\let\s=\sigma
\let\f=\phi
\let\vf=\varphi
\let\p=\psi
\let\om=\omega
\let\G= \Gamma
\let\D=\Delta
\let\Lam=\Lambda
\let\S=\Sigma
\let\Om=\Omega
\let\wt=\widetilde
\let\wh=\widehat
\let\convf=\leftharpoonup
\let\tri\triangle

%LETTRES RONDES
\def\cA{{\mathcal A}}
\def\cB{{\mathcal B}}
\def\cC{{\mathcal C}}
\def\cD{{\mathcal D}}
\def\cE{{\mathcal E}}
\def\cF{{\mathcal F}}
\def\cG{{\mathcal G}}
\def\cH{{\mathcal H}}
\def\cI{{\mathcal I}}
\def\cJ{{\mathcal J}}
\def\cK{{\mathcal K}}
\def\cL{{\mathcal L}}
\def\cM{{\mathcal M}}
\def\cN{{\mathcal N}}
\def\cO{{\mathcal O}}
\def\cP{{\mathcal P}}
\def\cQ{{\mathcal Q}}
\def\cR{{\mathcal R}}
\def\cS{{\mathcal S}}
\def\cT{{\mathcal T}}
\def\cU{{\mathcal U}}
\def\cV{{\mathcal V}}
\def\cW{{\mathcal W}}
\def\cX{{\mathcal X}}
\def\cY{{\mathcal Y}}
\def\cZ{{\mathcal Z}}

%MACROS SANS ARGUMENTS
\def\virgp{\raise 2pt\hbox{,}}
\def\cdotpv{\raise 2pt\hbox{;}}
\def\eqdef{\buildrel\hbox{\footnotesize d\'ef}\over =}
\def\eqdefa{\buildrel\hbox{\footnotesize def}\over =}
\def\Id{\mathop{\rm Id}\nolimits}
\def\limf{\mathop{\rm limf}\limits}
\def\limfst{\mathop{\rm limf\star}\limits}
\def\sgn{\mathop{\rm sgn}\nolimits}
\def\RE{\mathop{\Re e}\nolimits}
\def\IM{\mathop{\Im m}\nolimits}
\def\im {\mathop{\rm Im}\nolimits}
\def\Sp{\mathop{\rm Sp}\nolimits}
\def\C{\mathop{\mathbb C\kern 0pt}\nolimits}
\def\DD{\mathop{\mathbb D\kern 0pt}\nolimits}
\def\EE{\mathop{\mathbb E\kern 0pt}\nolimits}
\def\K{\mathop{\mathbb K\kern 0pt}\nolimits}
\def\N{\mathop{\mathbb  N\kern 0pt}\nolimits}
\def\Q{\mathop{\mathbb  Q\kern 0pt}\nolimits}
\def\R{{\mathop{\mathbb R\kern 0pt}\nolimits}}
\def\SS{\mathop{\mathbb  S\kern 0pt}\nolimits}
\def\St{\mathop{\mathbb  S\kern 0pt}\nolimits}
\def\Z{\mathop{\mathbb  Z\kern 0pt}\nolimits}
\def\ZZ{{\mathop{\mathbb  Z\kern 0pt}\nolimits}}
\def\H{{\mathop{{\mathbb  H\kern 0pt}}\nolimits}}
\def\PP{\mathop{\mathbb P\kern 0pt}\nolimits}
\def\TT{\mathop{\mathbb T\kern 0pt}\nolimits}
 \def\L {{\rm L}}

\def\h {{\rm h}}
\def\v {{\rm v}}

\newcommand{\ds}{\displaystyle}
\newcommand{\la}{\lambda}
\newcommand{\hn}{{\bf H}^n}
\newcommand{\hnn}{{\mathbf H}^{n'}}
\newcommand{\ulzs}{u^\lam_{z,s}}
\def\bes#1#2#3{{B^{#1}_{#2,#3}}}
\def\pbes#1#2#3{{\dot B^{#1}_{#2,#3}}}
\newcommand{\ppd}{\dot{\Delta}}
\def\psob#1{{\dot H^{#1}}}
\def\pc#1{{\dot C^{#1}}}
\newcommand{\Hl}{{{\mathcal  H}_\lam}}
\newcommand{\fal}{F_{\al, \lam}}
\newcommand{\Dh}{\Delta_{{\mathbf H}^n}}
\newcommand{\car}{{\mathbf 1}}
\newcommand{\X}{{\mathcal  X}}
\newcommand{\fgl}{F_{\g, \lam}}

%%%%%%%
\newcommand{\andf}{\quad\hbox{and}\quad}
\newcommand{\with}{\quad\hbox{with}\quad}
%%%%%%%

%%%%%%%
\def\beginproof {\noindent {\it Proof. }}
\def\endproof {\hfill $\Box$}

%%%%%%%%%%%%%

\def\vp{{\underline v}}
\def\presspO{{{\underline p}_0}}
\def\presspun{{{\underline p}_1}}
\def\wp{{\underline w}}
\def\wpe{{\underline w}^{\e}}
\def\vapp{v_{app}^\e}
\def\vapph{v_{app}^{\e,h}}
\def\vbar{\overline v}
\def\barEE{\underline{\EE}}

%%%%%%%%%%%%

\def\demo{d\'e\-mons\-tra\-tion}
\def\dive{\mathop{\rm div}\nolimits}
\def\curl{\mathop{\rm curl}\nolimits}
\def\cdv{champ de vec\-teurs}
\def\cdvs{champs de vec\-teurs}
\def\cdvdivn{champ de vec\-teurs de diver\-gence nul\-le}
\def\cdvdivns{champs de vec\-teurs de diver\-gence
nul\-le}
\def\stp{stric\-te\-ment po\-si\-tif}
\def\stpe{stric\-te\-ment po\-si\-ti\-ve}
\def\reelnonentier{\R\setminus{\bf N}}
\def\qq{pour tout\ }
\def\qqe{pour toute\ }
\def\Supp{\mathop{\rm Supp}\nolimits\ }
\def\coinfty{in\-d\'e\-fi\-ni\-ment
dif\-f\'e\-ren\-tia\-ble \`a sup\-port com\-pact}
\def\coinftys{in\-d\'e\-fi\-ni\-ment
dif\-f\'e\-ren\-tia\-bles \`a sup\-port com\-pact}
\def\cinfty{in\-d\'e\-fi\-ni\-ment
dif\-f\'e\-ren\-tia\-ble}
\def\opd{op\'e\-ra\-teur pseu\-do-dif\-f\'e\-ren\tiel}
\def\opds{op\'e\-ra\-teurs pseu\-do-dif\-f\'e\-ren\-tiels}
\def\edps{\'equa\-tions aux d\'e\-ri\-v\'ees
par\-tiel\-les}
\def\edp{\'equa\-tion aux d\'e\-ri\-v\'ees
par\-tiel\-les}
\def\edpnl{\'equa\-tion aux d\'e\-ri\-v\'ees
par\-tiel\-les non li\-n\'e\-ai\-re}
\def\edpnls{\'equa\-tions aux d\'e\-ri\-v\'ees
par\-tiel\-les non li\-n\'e\-ai\-res}
\def\ets{espace topologique s\'epar\'e}
\def\ssi{si et seulement si}

% MACRO EN ANGLAIS SANS ARGUMENTS
\def\pde{partial differential equation}
\def\iff{if and only if}
\def\stpa{strictly positive}
\def\ode{ordinary differential equation}
\def\coinftya{compactly supported smooth}

%%%%%%%%%%%%%%%%%%%%%%%%%%

  \title[Non linear estimate on the life span of (NS)]
  { A non linear estimate on the life span of solutions  \\
  of the three dimensional Navier-Stokes equations
  }

\author[J.-Y. Chemin]{Jean-Yves  Chemin}
\address[J.-Y. Chemin]%
{Laboratoire Jacques Louis Lions - UMR 7598,  Sorbonne Universit\'e \\ Bo\^\i te courrier 187, 4 place Jussieu, 75252 Paris
Cedex 05, France}
\email{chemin@ann.jussieu.fr }
\author[I. Gallagher]{Isabelle Gallagher}
\address[I. Gallagher]%
{DMA, \'Ecole normale sup\'erieure, CNRS, PSL Research University, 75005 Paris 
 \\
and UFR de math\'ematiques, Universit\'e Paris-Diderot, Sorbonne Paris-Cit\'e, 75013 Paris, France.}
\email{gallagher@math.ens.fr}
\subjclass[2010]{}
\begin{abstract}
The purpose of this article is to establish bounds from below for the life span of regular solutions to the incompressible Navier-Stokes system, which
involve norms not only of the initial data, but also of nonlinear functions of  the initial data. We provide examples showing that those bounds are significant improvements to the one provided by the classical fixed point argument. One of the important ingredients   is the use of  a scale-invariant energy estimate.
  \end{abstract}

 \maketitle
 
\section{Introdution}
In this article our aim is to give bounds from below for the life span of   solutions to the incompressible Navier-Stokes system in the whole space~$\R^3$. We are not interested here  in the regularity of the initial data: we focus on obtaining bounds from below for the life span associated with regular initial data.  Here regular means that the initial data belongs to the intersection of   all Sobolev spaces of non negative index. Thus all the solutions we consider are regular ones, as long as they exist.

Let us recall   the incompressible Navier-Stokes system, together with some of its basic features. The incompressible Navier-Stokes system  is the following:
$$
(NS) \quad \left\{
\begin{array}{c}
\partial_t u  -\D u + u\cdot \nabla u  =  -\nabla p\\
\dive u = 0\andf u_{|t=0} =u_0\,,
\end{array}
\right.
$$
where~$u$ is a three dimensional, time dependent vector field and~$p$  is the pressure, determined by the incompressibility condition~$\dive u=0$:
$$
-\Delta p =\dive  (u\cdot \nabla u) = \sum_{1 \leq i,j\leq 3} \partial_i\partial_j (u^i u^j) \,. 
$$
  This system has two fundamental properties related to its physical origin:
\begin{itemize}
\item 
  scaling invariance
\item
  dissipation of   kinetic energy.
\end{itemize}

The scaling property is the fact that if a function~$u$ satisfies~$(NS)$  on a time interval~$[0,T]$ with the initial data~$u_0$, then
the function~$u_\lam$ defined by
$$
u_\lam(t,x)\eqdefa \lam u(\lam^2 t,\lam x)
$$
satisfies~Ê$(NS)$ on the time interval~Ê$[0,\lam^{-2} T]$ with the initial data~$\lam u_0(\lam\,\cdot)$. This property is far from being a characteristic property of the system~$(NS)$. It is indeed satisfied by all   systems of the form
$$
(GNS)\quad \left\{
\begin{array}{c}
\partial_t u  -\D u +Q(u,u)  =  0\\
 u_{|t=0} =u_0
\end{array}
\right.\with Q^i (u,u) \eqdefa \sum_{1\leq j,k\leq 3} A^i_{j,k}(D) (u^ju^k)
$$
where the~$A^i_{j,k}(D)$ are smooth homogenenous Fourier multipliers of order~$1$.  Indeed denoting by~${\mathbb P}$ the projection onto divergence free vector fields
$$
{\mathbb P}\eqdefa \mbox{Id} - (\partial_i\partial_j \Delta^{-1})_{ij}
$$
the Navier-Stokes system takes   the form
$$
  \left\{
\begin{array}{c}
\partial_t u  -\D u +{\mathbb P}\dive (u\otimes u)  =  0\\
 u_{|t=0} =u_0\,,
\end{array}
\right.$$
which is of the type~(GNS).
For this class of systems,  the following result holds. The definition of homogeneous Sobolev spaces~$\dot H^s$ is recalled in the Appendix.   
\begin{prop}
\label {wpgnsHs}
{\sl
Let~$u_0$ be a regular three dimensional vector field. A positive time~$T$ exists such that a unique regular solution to (GNS) exists on~$[0,T]$.  Let~$T^\star(u_0)$ be the maximal time of existence of this regular solution.  Then, for any~$\g$ in the interval~$]0,1/2[$, a constant~$c_\g$ exists such that
\begin{equation}\label{lowerboundfpProp}
T^\star(u_0) \geq c_\g \|u_0\|_{\dot H^{\frac 12 +2\g}}^{-\frac 1 \g} \,.
\end{equation}
}
\end{prop}
In the case when~Ê$\g=1/4$  for the particular case of~$(NS)$, this type of result goes back to the seminal work of J. Leray (see\ccite{lerayns}). Let us point out that the same type of result can be proved for the~$L^{3+\frac {6\g} {1-2\g}}$ norm.
\begin{proof}
This result is obtained by a scaling argument. Let us define the following function
$$
\underline T_{\dot H^{\frac 12+2\g}} (r) \eqdefa \inf\bigl\{T^\star(u_0)\,,\ \|u_0\|_{\dot H^{\frac 12+2\g}}=r\bigr\} \,.
$$
We assume that at least one smooth initial data~$u_0$ develops singularites, which means exactly that~$T^\star(u_0)$ is finite.  Let us mention that this lower bound is in fact a minimum (see\ccite{poulon}). Actually the function~$\underline T_{\dot H^{\frac 12+2\g}}$ may be computed using a scaling argument. Observe that 
$$
\|u_0\|_{\dot H^{\frac 12+2\g}} =  r\Longleftrightarrow \|r^{-\frac 1 {2\g}} u_0(r^{-\frac 1 {2\g}}\cdot) \|_{\dot H^{\frac 12+2\g}} = 1\,.
$$ 
As we have~$T^\star(u_0) = r^{-\frac 1 \g} T^\star \big( r^{-\frac 1 {2\g}} u_0(r^{-\frac 1 {2\g}}\cdot)\big)$, we infer that~$
\underline T_{\dot H^{\frac 12+2\g}} (r) = r^{-\frac 1 \g} \underline T_{\dot H^{\frac 12+2\g}} (1)$
and thus  that
$$
T^\star(u_0)\geq c_\g \|u_0\|_{\dot H^{\frac 12+2\g}}^{-\frac 1\g} \with c_\g \eqdefa \underline T_{\dot H^{\frac 12+2\g}} (1)\,.
$$
The proposition is proved.
\end{proof}

\bigskip

Now let us investigate  the optimality of such a result, in particular  concerning the norm appearing in the lower bound~(\ref{lowerboundfpProp}).   Useful results and definitions concerning  Besov spaces are recalled in the Appendix; 
the  Besov norms of particular interest in this text are the~$\dot B^{-1}_{\infty,2}$ norm which is given by
$$
\|a\|_{\dot B^{-1}_{\infty,2}}\eqdefa \Big(\int_0^\infty \|e^{t\D} a\|^2_{L^\infty}\, dt\Big)^\frac12
$$
and the Besov norms~$\dot B^{-\s}_{\infty,\infty}$ for~$\sigma>0$ which are  
$$
\|a\|_{\dot B^{-\s}_{\infty,\infty}} \eqdefa \sup_{t>0} t^{\frac \s 2} \|e^{t\D} a\|_{L^\infty}\,.
$$
It has been known since~\cite{fujitakato} that a smooth  initial data in~$\dot H^{\frac 12}$ (corresponding of course to the limit case~$\gamma = 0$ in Proposition~\ref{wpgnsHs}) generates a smooth solution for some time~$T>0$. 
 Let us point out that in dimension~$3$,   the following inequality holds
$$
\|a\|_{\dot B^{-1} _{\infty,2}} \lesssim  \|a\|_{\dot H^{\frac 12}}.
$$
The norms~$\dot B^{-\s}_{\infty,\infty}$ are the smallest norms invariant by translation and having a given scaling. More precisely, we have the following result, due to Y. Meyer (see Lemma 9 in\ccite{meyerNSlivre}).
\begin{prop}
\label{besovendpoint}
{\sl
Let~$d\geq 1$ and let~$(E,\|\cdot\|_E)$ be a normed space continuously included in~$\cS'(\R^d)$,  the space of tempered distributions on~$\R^d$. Assume that~$E$ is stable by translation and by dilation, and that a constant~$C_0$ exists such that  
$$
\forall (\lam,e) \in ]0,\infty[\times \R^d\,,\ \forall a \in E\,,\quad  \|a(\lam\cdot -e)\|_{E} \leq C_0\lam^{-\s} \|a\|_{E}\,.
$$ 
Then a constant~$C_1$ exists such that 
$$
\forall a \in E\,,\  \|a\|_{\dot B^{-\al}_{\infty,\infty}}  \leq C_1 \|a\|_{E}\,.
$$
}
\end{prop}

\begin{proof}
Let us simply observe that, as~$E$ is continuously included in~$\cS'(\R^d)$,   a constant~$C$ exists such that for all~$a $ in~$ E$,
$$
\bigl | \langle a, e^{-|\cdot|^2 } \rangle \bigr | \leq C\|a\|_{E}\,.
$$
Then  by invariance by translation and dilation of~$E$, we infer immediately that
$$
\|e^{t\D} a\|_{L^\infty} \leq C_1 t^{-\frac \s 2} \|a\|_E
$$
which proves the proposition.
\end{proof}

\medbreak

Now let us state a first improvement  to Proposition\refer {wpgnsHs} where the life span is bounded from below in terms of the~$\dot B^{-1+2\g}_{\infty,\infty}$ norm of the initial data. 
\begin{theo}
\label {wpgnsBesov}
{\sl
With the notations of Proposition{\rm\refer {wpgnsHs}},  for any~$\g$ in the interval~$]0,1/2[$, a constant~$c'_\g$ exists such that
\begin{equation}\label{lowerboundfpThm}
T^\star(u_0) \geq  T_{\rm FP, \g} (u_0) \eqdefa  c'_\gamma \|u_0\|_{\dot B^{-1+2\gamma}_{\infty,\infty}}^{-\frac 1\gamma}\,.
\end{equation}
}
\end{theo}
This theorem is proved in Section\refer {proofclassicrevisted}; the proof relies on a fixed point theorem in a space included in the space of~$L^2$ in time functions, with values in~$L^\infty$.

\medskip

Let us also recall that if a scaling~$0$ norm of a regular initial data is small, then the solution of~$(NS)$ associated with~$u_0$ is global. This a consequence of the Koch and Tataru theorem (see\ccite {kochtataru}) which can be translated  as  follows in the context of smooth solutions.
\begin{theo}
\label {kochtatarusmooth}
{\sl
A constant $c_0$ exists such that for any regular initial data~$u_0$ satisfying
$$
\|u_0\|_{BMO^{-1}} \eqdefa \sup_{t>0} t^{\frac 12} \|e^{t\Delta} u_0\|_{L^\infty} 
+\Big(\supetage {x\in \R^3} {R>0} \frac 1 {R^3} \int_0^{R^2} \int_{B(x,R)}|
e^{t\D}u_0(y) |^2dy dt \Big)^\frac12\leq c_0\,,
$$
the associate solution of~$(GNS)$ is globally regular.
}
\end{theo}

Let us remark that 
$$
\|u_0\|_{\dot B^{-1}_{\infty,\infty}} \leq  \|u_0\|_{BMO^{-1}} \leq \|u_0\|_{\dot B^{-1}_{\infty,2}}.
$$

We shall explain in Section\refer {proofclassicrevisted} how to deduce Theorem~\ref{kochtatarusmooth} from the Koch and Tataru theorem\ccite {kochtataru}.

\medbreak

The previous results are valid for the whole class of systems~$(GNS)$. Now let us present the second main feature of the incompressible Navier-Stokes system, which is not 
shared by all systems under the form~$(GNS)$ as it relies on a special structure of the nonlinear term (which must be skew-symmetric in~$L^2$): the dissipation estimate for the kinetic energy. For regular solutions of~$(NS)$ there holds
$$
\frac 1 2 \frac d {dt} \|u(t)\|_{L^2}^2 +\|\nabla u(t)\|^2_{L^2} =0
$$
which gives by integration in time
\beq
\label {energydissipationfond}
\forall t \geq 0\, , \quad \cE  \big(u(t)\big )\eqdefa  \frac 12 \|u(t)\|_{L^2}^2 +\int_0^t\|\nabla u(t')\|_{L^2} dt'  = \frac 12 \|u_0\|_{L^2}^2\,.
\eeq
%To the best of our knowledge, this identity is not used in problems involving the life span of regular solutions; 
T. Tao pointed out in his paper\ccite {taoNSmodif}  that the energy estimate is not enough to prevent possible singularities from appearing. Our purpose  here is to investigate if this energy estimate can improve the lower bound~(\ref{lowerboundfpThm}) of the life span for regular initial data.
We recall indeed that for   smooth initial data, all Leray solutions --- meaning solutions in the sense of distributions satisfying the energy inequality
\beq
\label{energyinequality}
\cE  \big(u(t)\big ) \leq  \frac 12 \|u_0\|_{L^2}^2
\eeq
coincide with the smooth solution as long as the latter exists.

What we shall use  here is a rescaled version  of the energy dissipation inequality in the spirit of\ccite  {CPNS2011},  on the fluctuation~$w\eqdefa u-u_\L$ with~$u_{\L} (t)\eqdefa  e^{t\Delta }u_0$.
\begin{prop}
\label{rescaledenergy}
{\sl Let~$u$ be a regular solution of~$(NS)$ associated with some initial data~$u_0$. Then the fluctuation~$w$ satisfies, for any positive~$t$%The fluctuation~$w$ satisfies the following global in time estimate:
$$
\cE \Big(
 \frac{ w(t)}{t^{\frac14}}
\Big) + \int_0^t  \frac { \|w(t')\|_{L^2}^2} { {t'}^{\frac 32}}\, dt' \lesssim  Q_{\L}^0  \, \exp \|u_0\|^ 2_{\dot B^{-1}_{\infty,2}}\,\with\, Q_{\L}^0  \eqdefa \,\int_0^\infty \,t^\frac12\| \PP (u_\L\cdot\nabla u_\L)(t)\|_{L^2}^2 \, dt
 \,.
$$
}
\end{prop}
Our main result is then the following.
\begin{theo}\label{theolifespan}
{\sl
There is a constant~$C>0$ such that the following holds. For any regular initial data of~$(NS)$, 
\begin{equation}
\label {lowerboundbetter}
T^*(u_0) > T_\L(u_0) \eqdefa C \big(Q_{\L}^0\big)^ {-2} \Big(
\|\partial_3u_0\|_{\dot B^{-\frac32}_{\infty,\infty}}^2 Q_{\L}^0+ \sqrt{Q_{\L}^0Q_{\L}^1}
\Big)^ {-2}  \exp \big(-4
\|u_0\|^ 2_{\dot B^{-1}_{\infty,2}}
\big)\,,
\end{equation}
with
$$
 Q_{\L}^1  \eqdefa \int_0^\infty t^\frac32\big\|\partial_3^2 \big(\PP (u_\L\cdot\nabla u_\L) \big)(t)\big\|_{L^2}^2 \, dt\,.
$$
}
\end{theo}
The main two features of this result are that
\begin{itemize}
\item the statement involves non linear quantities associated with the initial data, namely norms of~$\PP (u_\L\cdot\nabla u_\L)$;
\item one particular (arbitrary) direction plays a specific role.
\end{itemize}
This theorem is proved in Section~\ref{prooftheolifespan}.

\medbreak
%Let us compare this lower bound for the life span with the lower bound given by Theorem\refer   {wpgnsBesov}.
The following statement shows that  the lower bound on~$T^*(u_0)$ given in Theorem~\ref{theolifespan} is, for some classes of initial data, a significant improvement.

   \begin{theo}\label{theoexample}
{\sl Let~$(\gamma,\eta)$ be in~$]0,1/2[\times ]0,1[$.  There is a constant~$C$ and a family~$(u_{0,\varepsilon})_{\varepsilon\in ]0,1[}$ of regular initial data  such that 
with the notation of  Theorems~{\rm\ref{wpgnsBesov}} and~{\rm\ref{theolifespan}},  
 $$
  T_{\rm FP} (u_{0,\varepsilon})= C   \varepsilon^{2 }  |\log \e|^{-\frac1\gamma}\quad \mbox{and} \quad T_\L (u_{0,\varepsilon}) \geq   C \varepsilon^{-2+\eta} \, .
$$
}
\end{theo}

This theorem is proved in Section~\ref{example}. The family~$(u_{0,\varepsilon})_{\varepsilon\in ]0,1[}$ is closely related to the family used in\ccite {cg2} to exhibit families of initial data which do not obey   the hypothesis of the Koch and Tataru theorem and which nevertheless generate global smooth solutions. However it it too large to   satisfy the assumptions of Theorem~2 in\ccite {cg2} so it is not known if the associate solution is global.

\medskip

In the following we shall denote by~$C$ a constant which may change from line to line, and we shall sometimes write~$A \lesssim B$ for~$A \leq CB$.
\section{Proof of Theorem~\refer{wpgnsBesov}}
\label {proofclassicrevisted}
Let~$u_0$ be a smooth vector field
and let us solve~$(GNS)$ by means of a fixed point method. We define the bilinear operator~$B$ by
\beq
\label {wpgnsBesovdemoeq0}
\begin{aligned}
\partial_t B(u,v) -\D B(u,v) = -\frac 12 \bigl(Q(u,v)+Q(v,u) \bigr)\, ,  \quad  \mbox{and}\   B(u,v)|_{t=0}=0\,.
\end{aligned}
\eeq
One can decompose the solution~$u$ to~$(GNS)$ into
$$
u = u_{\L} + B(u,u)\,.
$$
Resorting to the Littlewood-Paley decomposition defined in the Appendix, let us define for any real number~$\gamma$ and any time~$T>0$, the quantity
$$
\|f\|_{E^\gamma_T} \eqdefa \sup_{j\in {\mathbb Z}} 2^{-j(1-2\gamma)} \bigl(\|\D_jf\|_{L^\infty([0,T]\times \R^3)} +2^{2j} \|\D_jf\|_{L^1([0,T]; L^\infty(\R^3))}\bigr)\, .
$$
Using Lemma 2.1 of\ccite {chemin20} it is easy to see that
$$
\|u_\L\|_{E^\g_\infty} \lesssim \|u_0\|_{\dot B^{-1+2\gamma}_{\infty,\infty}}\, ,
$$
so Theorem~\ref{wpgnsBesov} will follow from the fact that~$B$ maps~$E^\gamma_T\times E^\gamma_T$ into~$E^\gamma_T$ with   the following estimate:
\beq
\label{lifespanNSB-1+demoeq1}
\|B(u,v)\|_{E^\g_T} \leq C_\g T^{  \g} \|u\|_{E^\g_T} \|v\|_{E^\g_T} \, .
\eeq
So let us prove~(\ref{lifespanNSB-1+demoeq1}).
Using again Lemma 2.1 of\ccite {chemin20} along with the fact that the~$A^i_{k,\ell} (D) $ are smooth homogeneous Fourier multipliers of order~$1$, we have 
$$
\|\D_j B(u,v) (t) \|_{L^\infty} \lesssim  \int_0^t  e^{-c2^{2j(t-t')} } 2^j \Bigl\|\D_j \big(u(t')\otimes v(t')+v(t')\otimes u(t')\Big)\|_{L^\infty} dt'\,.
$$
We then decompose (component-wise) the product~$u\otimes v$ following Bony's paraproduct algorithm: for all functions~$a$ and~$b$
the support of the Fourier transform  of~$S_{j'+1} a\D_{j'} b$ and~$S_{j'} b\D_{j'} a$ is included in a ball~$2^{j'} B$ where~$B$ is a fixed ball of~$\R^3$,
so one can write for some fixed constant~$c>0$
$$
ab = \sum_{2^{j'} \geq c2^ j} \big( S_{j'+1}a \Delta_{j'}b +\Delta_{j'} a S_{j'} b \big)
$$
so thanks to Young's inequality in time one can write
\beq
\label{lifespanNSB-1+demoeq2}
\begin{split}
&2^{-j(1-2\gamma)} \bigl(\|\D_jB(u,v)\|_{L^\infty([0,T]\times \R^3)} +2^{2j} \|\D_jB(u,v)\|_{L^1([0,T]; L^\infty(\R^3))}\bigr)\\
&\qquad\qquad\qquad\qquad\qquad\qquad\qquad\qquad\qquad\qquad\qquad{}\lesssim \cB^1_j(u,v)+\cB^2_j(u,v)\with \\
& \cB^1_j(u,v) \eqdefa 2^{2j\g}\! \! \sum_{2^{j'} \geq \max \{c2^j,T^{-\frac 12}\}} \! \! \|S_{j'+1}u\|_{L^\infty([0,T]\times \R^3)}
\|\D_{j'}v\|_{L^1([0,T]; L^\infty(\R^3))}\\
&\qquad\qquad\qquad\qquad
{}+2^{2j\g} \! \! \sum_{c2^j \leq2^{j'} <T^{-\frac 12}}\! \!  \|S_{j'+1}u\|_{L^\infty([0,T]\times \R^3)}
\|\D_{j'}v\|_{L^1([0,T]; L^\infty(\R^3))}\andf\\
& \cB^2_j(u,v) \eqdefa 2^{2j\g}\! \!  \sum_{2^{j'} \geq \max \{c2^j,T^{-\frac 12}\}} \! \! \|S_{j'}v\|_{L^\infty([0,T]\times \R^3)}
\|\D_{j'}u\|_{L^1([0,T]; L^\infty(\R^3))}\\
&\qquad\qquad\qquad\qquad
{}+2^{2j\g}  \! \! \sum_{c2^j \leq2^{j'} <T^{-\frac 12}} \! \! \|S_{j'}v\|_{L^\infty([0,T]\times \R^3)}
\|\D_{j'}u\|_{L^1([0,T]; L^\infty(\R^3))} \, .
\end{split}
\eeq
In each of the sums over~$c2^j \leq2^{j'} <T^{-\frac 12}$ we write
$$
\|f\|_{L^1([0,T]; L^\infty(\R^3))} \leq T \|f \|_{L^\infty([0,T]\times \R^3)}
$$
and we can estimate the two terms~$\cB^1_j(u,v)$ and~$\cB_j^2(u,v)$  in the same way: for~$\ell\in \{1,2\}$ there holds indeed
\beno
\cB_j^\ell(u,v) & \leq & \|u\|_{E^\g_T}\|v\|_{E_T^\g} \Bigl(2^{2j\gamma} \! \!   \sum_{2^{j'} \geq \max \{c2^{j},T^{-\frac 12}\}} \! \!  2^{-4j'\g}
+ T 2^{2j(1-\g)}\sum_{c\leq2^{j'-j} <(2^{2j}T)^{-\frac 12}} 2^{2(j'-j)(1-2\g)} \Bigr)\\
 &\leq &  \|u\|_{E^\g_T}\|v\|_{E_T^\g} \Bigl( T^\g
+ T 2^{2j(1-\g)}\! \! \sum_{c\leq2^{j'-j} <(2^{2j}T)^{-\frac 12}}\! \!  2^{2(j'-j)(1-2\g)} \Bigr)\,.
\eeno
Once noticed that
$$
T 2^{2j(1-\g)}\! \!\sum_{c\leq2^{j'-j} <(2^{2j}T)^{-\frac 12}} \! \!2^{2(j'-j)(1-2\g)}  \leq {\bf 1}_{\{2^{2j}T\leq C\}}  (T2^{2j}) ^\g 2^{-2j\g} \lesssim T^\g
$$
the estimate~(\ref{lifespanNSB-1+demoeq1})   is proved  and   Theorem~\refer{wpgnsBesov} follows.
\qed

\section {Proof of Theorem\refer {kochtatarusmooth} }
As the solutions given by the Fujita-Kato theorem~\cite{fujitakato} and  the  Koch-Tataru theorem~\cite{kochtataru}  are unique in their own class, they are unique in the intersection and thus coincide as long as the Fujita-Kato solution exists.  Thus  Theorem\refer {kochtatarusmooth}  is a question of propagation of regularity, which is provided by the following lemma (which proves the theorem).
\begin{lemme}
\label {LemmaFK--KT}
{\sl
 A constant~$c_0$ exists which satisfies the following.
Let~$u$ be a regular solution of~$(GNS)$ on~$[0,T [$ associated with a regular initial data~$u_0$ such that 
$$
\|u\|_{\rm K}\eqdefa \sup_{t\in [0,  T[} t^{\frac 12} \|u(t)\|_{L^\infty} \leq c_0\,.
$$
Then~$T^\star(u_0)>  T$.
}
\end{lemme}

\begin{proof}
The proof is based on a paralinearization argument (see\ccite {chemin20}).  Observe that for any~$T$ less than~$T^\star(u_0)$,~$u$  is a solution on~$[0,T[$ of the  {\it linear} equation
$$
\displaylines{
(PGNS)\quad \left\{
\begin{array}{c}
\ds \partial_t v  -\D v +\cQ (u,v)  =  0\\
 v_{|t=0} =u_0 
\end{array}
\right. \with  \cr 
\cQ (u,v)\eqdefa  \sum_{j\in \ZZ} Q(S_{j+1} u,\D_jv)+ \sum_{j\in \ZZ} Q(\D_j v,S_ju)
\,.}
$$
In the same spirit as\refeq {wpgnsBesovdemoeq0}, let us define~$PB(u,v)$ by
\beq
\label {kochtatarusmoothdemoeq1}
\partial_t PB(u,v) -\D PB(u,v) =-\cQ(u,v)  \, \quad \mbox{and}\quad PB(u,v)|_{t=0}=0\,.
\eeq
A solution  of~$(PGNS)$ is a solution of
$$
v = u_\L +PB(u,v)\,.
$$
Let us introduce the space~$F_T$ of continuous functions with values in~$\dot H^{\frac 12}$, which are elements of~$L^4([0,T];\dot H^1)$, equipped with the norm
$$
\|v\|_{F_T} \eqdefa \biggl( \sum_{j\in \ZZ}  2^{j} \| \D_j v\|^2_{L^\infty([0,T[;L^2)} \biggr)^{\frac 12} +
\|v\|_{L^4([0,T[; \dot H^1)}\,.
$$
Notice that the first part of the norm was   introduced in\ccite {chemin13} and is a larger norm than the supremum in time of the~$\dot H^{\frac 12}$ norm. Moreover there holds
$$
\| u_\L\|_{F_T} \lesssim\|u_0\|_{\dot H^{\frac 12}} \,.
$$
Let us admit for a while the following inequality:
\beq
\label {kochtatarusmoothdemoeq2}
\|PB(u,v)\|_{F_T}  \lesssim  \|u\|_{\rm K} \|v\|_{F_T}\,.
\eeq
Then it is obvious that if~$ \|u\|_{\rm K}$ is small enough for some time~$[0,  T[$, the linear equation~$(PGNS)$ has  a unique solution in~$F_{  T}$ (in the distribution sense) which satisfies in particular,  if~$c_0$ is small enough,
\beno
\|v\|_{F_{  T}} \leq C\|u_0\|_{\dot H^{\frac 12}} + \frac 12 \|v\|_{F_{  T}} \,.
\eeno
 As~$u$ is a regular solution of~$ (PGNS)$, it therefore satisfies  
  $$
\forall t<  T\,,\  \|u\|_{L^4([0,t];\dot H^1)} \leq 2 C\|u_0\|_{\dot H^{\frac 12}}
 $$
 which implies that~$T^\star (u_0)>  T$ ,  so the lemma is proved  provided we prove Inequality\refeq {kochtatarusmoothdemoeq2}. 
 
 Let us observe that for any~$j$ in~$\ZZ$, 
\beq
\label {kochtatarusmoothdemoeq3}
\partial_t \D_j PB(u,v) -\D \D_j PB(u,v) =-\D_j\cQ(u,v)\,.
 \eeq
 By definition of~$\cQ$, we have
 $$
\bigl\| \D_j\cQ(u,v) (t) \|_{L^2}  \leq \sum_{j'\in \ZZ} \sum_{1\leq i, k,\ell\leq 3} \bigl (\bigl \| \D_jA^i_{k,\ell} (D) \bigl(S_{j'+1} u\D_{j'} v\bigr)\bigr\|_{L^2} + \bigl \| \D_jA^i_{k,\ell} (D) \bigl(\D_{j'} vS_{j'} u\bigr)\bigr\|_{L^2}\bigr)\,.
 $$
 As~$A^i_{k,\ell} (D) $ are smooth homogeneous Fourier multipliers of order~$1$, we   infer that  for some fixed nonnegative integer~$N_0 $
\beno
\bigl\| \D_j\cQ(u,v) (t) \|_{L^2} &   \lesssim&  2^j \sum_{j'\geq j-N_0}  \bigl (\bigl \| S_{j'+1} u(t)\D_{j'} v(t)\bigr\|_{L^2} + \bigl \| \D_{j'} v(t)S_{j'} u(t)\bigr\|_{L^2}\bigr)\\
&   \lesssim &  2^j \sum_{j'\geq j-N_0}  \bigl ( \| S_{j'+1} u(t)\|_{L^\infty} \|\D_{j'} v(t)\|_{L^2} + \bigl \| \D_{j'} v(t)\|_{L^2} \|S_{j'} u(t)\bigr\|_{L^\infty}\bigr)\\
&   \lesssim &  2^j   \|u(t)\|_{L^\infty} \sum_{j'\geq j-N_0}   \|\D_{j'} v(t)\|_{L^2}\,.
\eeno
Using Relation\refeq {kochtatarusmoothdemoeq3} and the definition of the norm on~$F_T$,  we infer that 
\beno
\|\D_j PB(u,v)(t)\|_{L^2} & \leq  &  \int_0^t e^{-c2^{2j}(t-t')} \bigl\| \D_j\cQ(u,v) (t') \|_{L^2} dt'\\
& \lesssim & 2^j    \int_0^t e^{-c2^{2j}(t-t')}  \|u(t')\|_{L^\infty} \sum_{j'\geq j-N_0}   \|\D_{j'} v(t')\|_{L^2} dt'\\
 & \lesssim & 2^j \|u\|_{\rm K} \|v\|_{F_T}  \sum_{j'\geq j-N_0} c_{j'} 2^{-\frac {j'} 2} \int_0^t e^{-c2^{2j}(t-t')} \frac 1 {\sqrt {t'}} dt' \,,
 \eeno
 where~$(c_j)_{j\in \ZZ} $ denotes a generic element of the sphere of~$\ell^2(\ZZ)$. 
Thus we have, for all~$t$  less than~$T$, 
\beno
2^{\frac j 2} \|\D_j PB(u,v)(t)\|_{L^2} \lesssim \|u\|_{\rm K} \|v\|_{F_T}  \sum_{j'\geq j-N_0} c_{j'} 2^{-\frac {j'-j} 2} \int_0^t  2^j  e^{-c2^{2j}(t-t')} \frac 1 {\sqrt {t'}} dt' \,.
\eeno
Thanks to Young's inequality, we have~$\ds   \sum_{j'\geq j-N_0} c_{j'} 2^{-\frac {j'-j} 2}\lesssim c_j$ and we deduce that 
\beq
\label {kochtatarusmoothdemoeq4}
2^{\frac j 2} \|\D_j PB(u,v)(t)\|_{L^2} \lesssim  c_j \|u\|_{\rm K} \|v\|_{F_T}  \int_0^t  2^j  e^{-c2^{2j}(t-t')} \frac 1 {\sqrt {t'}} dt'. 
\eeq
As we have
\beno
\int_0^t  2^j  e^{-c2^{2j}(t-t')} \frac 1 {\sqrt {t'}} dt'   \lesssim  \int_0^t  \frac 1 {\sqrt {t-t'}}  \frac 1 {\sqrt {t'}} dt' \,,
\eeno we infer finally that 
\beq \label{kochtatarusmoothdemoeq5} 
\sum_{j\in \ZZ} 2^j \|\D_j PB(u,v)\|^2_{L^\infty([0,T]; L^2)} \lesssim   \|u\|^2_{\rm K} \|v\|_{F_T} ^2\,.
\eeq
Moreover returning to Inequality\refeq{kochtatarusmoothdemoeq4}, we have
$$
2^j \|\D_j PB(u,v)\|_{L^4([0,T]; L^2)} \lesssim  c_j \|u\|_{\rm K} \|v\|_{F_T} \Bigl\|  \int_0^t  2^{\frac {3j} 2}  e^{-c2^{2j}(t-t')} \frac 1 {\sqrt {t'}} dt'\Bigr\|_{L^4(\R+)}\, . 
$$
The Hardy-Littlewood-Sobolev inequality implies that 
$$
 \Bigl\|  \int_0^t  2^{\frac {3j} 2}  e^{-c2^{2j}(t-t')} \frac 1 {\sqrt {t'}} dt'\Bigr\|_{L^4(\R+)} \lesssim1\, .
$$
Since thanks to the Minkowski inequality there holds
$$
\| PB(u,v) \|_{L^4([0,T];\dot H^1)}^2 \leq \sum_{j\in \ZZ} 2^{2j}  \|\D_j PB(u,v)\|^2 _{L^4([0,T]; L^2)}\,,
$$
together with Inequality\refeq  {kochtatarusmoothdemoeq5} this concludes the proof of Inequality\refeq {kochtatarusmoothdemoeq2} and thus the proof of Lemma\refer {LemmaFK--KT}.
\end{proof}

\section{Proof of Theorem \ref{theolifespan}}\label{prooftheolifespan}
The plan of the proof of Theorem~\ref{theolifespan}
 is the following: 
 as previously we look for the solution of $(NS)$ under the form
$$
u = u_\L + w
$$
where we recall that~$ u_\L (t) = e^{t\Delta} u_0$. Moreover we recall that the solution~$u$ satisfies the energy inequality~(\ref{energyinequality}). By construction, the fluctuation~$w$ satisfies 
$$
{\rm(NSF)}\qquad\qquad \partial_t w-\Delta w +(u_{\L}+w)\cdot \nabla w +w\cdot \nabla u_{\L} = -u_{\L}\cdot \nabla u_{\L} -\nabla p\,, \quad \dive w = 0 \,.
$$
Let us prove that the life span of~$w$ satisfies the     lower bound~(\ref{lowerboundbetter}). 
The first step of the proof consists in proving Proposition~\ref{rescaledenergy}, stated in the introduction. This   is achieved in Section~\ref{proofrescaledenergy}.
The next step is the proof of a similar energy estimate on~$\partial_3 w$ --- note that contrary to
the scaled energy estimate of Proposition~\ref{rescaledenergy}, the next result
 is useful in general only locally in time. It is proved in Section~\ref{proofrenergypartial3}.
\begin{prop}\label{energypartial3}
{\sl With the notation of  Proposition~{\rm\ref{rescaledenergy}} and Theorem~{\rm\ref{theolifespan}}, the fluctuation~$w$ satisfies the following estimate:
$$
\cE \big(
\partial_3 w
\big)(t) \lesssim
\Big(Q_{\L}^0\big(  t^\frac12 \sup_{t'\in (0,t)}\|\partial_3 w(t)\|_{L^2}^4 + \|\partial_3u_0\|^2_{\dot B^{-\frac32}_{\infty,\infty}} \big)+ \sqrt{Q_{\L}^0 Q_{\L}^1 }
\Big)
\exp  \big( 2 \|u_0\|^ 2_{\dot B^{-1}_{\infty,2}}\big)
 \,.
$$
}\end{prop}
Combining both propositions, one can conclude the proof of Theorem~\ref{theolifespan}.  This is performed in Section~\ref{conclusionprooftheorem}. 

\subsection{The rescaled energy estimate on the fluctuation: proof of Proposition~\ref{rescaledenergy}}\label{proofrescaledenergy}
An~$L^2$ energy estimate on (NSF) gives
$$
\frac 12\frac d {dt} \|w(t)\|_{L^2}^2 + \|\nabla w(t)\|_{L^2}^2  = -\sum_{1\leq j,k\leq 3}\int_{\R^3} w^j\partial_j u_{\L}^k w^k(t,x) dx
-\bigl( \PP(u_{\L}\cdot \nabla u_{\L})\big | w\bigr)(t)\,.
$$
>From this, after an integration by parts and using the fact that the divergence of~$w$ is zero, we infer that 
$$
\longformule{\frac 12\frac d {dt}\Bigl( \frac { \|w(t)\|_{L^2}^2} { t^{\frac 12}} \Bigr)  + \frac { \|w(t)\|_{L^2}^2} { 2t^{\frac 32}} 
+  
\frac { \|\nabla w(t)\|_{L^2}^2} { t^{\frac 12}}  
}
{ {}\leq
\frac {\|w(t)\|_{L^2} \|u_{\L}(t)\|_{L^\infty} \|\nabla w(t)\|_{L^2} }  { t^{\frac 12}}
+\frac { \|\PP(u_{\L}\cdot \nabla u_{\L})(t)\|_{L^2}\|w(t)\|_{L^2}  } { t^{\frac 12}}\,\cdotp
}
$$
Let us observe that 
$$
\frac { \|\PP(u_{\L}\cdot \nabla u_{\L})(t)\|_{L^2}\|w(t)\|_{L^2}  } { t^{\frac 12}}
=   t^{\frac 14}  \|\PP(u_{\L}\cdot \nabla u_{\L})(t)\|_{L^2} 
\frac { \|w(t)\|_{L^2}  } { t^{\frac 34}}\,\cdotp
$$
Using a convexity inequality, we infer that 
$$
%\longformule{
\frac d {dt}\Bigl( \frac { \|w(t)\|_{L^2}^2} { t^{\frac 12}} \Bigr)  + \frac { \|w(t)\|_{L^2}^2} { 2t^{\frac 32}}+  
\frac { \|\nabla w(t)\|_{L^2}^2} { t^{\frac 12}}  
%}
%{{}
\leq
\frac {\|w(t)\|^2_{L^2} \|u_{\L}(t)\|^2_{L^\infty}  }  { t^{\frac 12}}
+ { t^{\frac 12}} \|u_{\L}(t)\cdot\nabla u_{\L}(t)\|^2_{L^2}\,.
%}
$$
Thus we deduce that
$$
\begin{aligned} \frac d {dt}\biggl( \frac { \|w(t)\|_{L^2}^2} { t^{\frac 12}}\exp \Bigl(- \!\!\int_0^t\!\!\!  \|u_{\L}(t')\|^2_{L^\infty} dt'   \Bigr)\biggr) 
 + \exp \Bigl(-  \!\!\int_0^t \!\!\! \|u_{\L}(t')\|^2_{L^\infty} dt'   \Bigr)\Bigl(\frac { \|w(t)\|_{L^2}^2} { 2t^{\frac 32}}+  
\frac { \|\nabla w(t)\|_{L^2}^2} { t^{\frac 12}}  \Bigr)
\\ \leq  \exp \Bigl(- \!\!\int_0^t \!\!\!\|u_{\L}(t')\|^2_{L^\infty} dt'   \Bigr)
 { t^{\frac 12}} \|\PP(u_{\L}\cdot \nabla u_{\L})(t)\|^2_{L^2}\,,
\end{aligned}
$$
from which we infer by the definition of the~$\dot B^{-1}_{\infty,2}$ norm and of~$Q_{\L}^0$ that
\beq
\label {estimfluctNL}
\forall t \geq 0\, , \quad   \frac {\|w(t)\|^2_{L^2}  }  { t^{\frac 12}}
+\int_0^t  \Bigl(\frac { \|w(t')\|_{L^2}^2} { 2{t'}^{\frac 32}}
+  
\frac { \|\nabla w(t)\|_{L^2}^2} {{t'}^{\frac 12}}  \Bigr) dt' 
\leq Q_{\L}^0 \, \exp {\|u_0\|^2_{\dot B^{-1}_{\infty,2}} } \, .
\eeq
Proposition~\ref{rescaledenergy} follows. \qed

\subsection{Proof of Proposition~\ref{energypartial3}}\label{proofrenergypartial3}

Now let us  investigate the evolution of~$\partial_3w$ in~$L^2$. Applying the partial differentiation~$\partial_3$ to~(NSF), we get
\begin{equation}\label{eqpartial3w}
\begin{aligned}
\partial_t \partial_3w&-\Delta \partial_3 w +(u_{\L}+w)\cdot \nabla  \partial_3 w +\partial_3w\cdot \nabla u_{\L} 
\\
&\quad = -\partial_3 u_{\L}\cdot \nabla  w - \partial_3 w\cdot \nabla w 
- w\cdot \nabla \partial_3 u_{\L} 
-\partial_3 ( u_{\L}\cdot \nabla u_{\L}) -\nabla \partial_3 p\,.
\end{aligned}
\end{equation}
The difficult terms to estimate are those which do not contain explicitly $\partial_3 w$. So let us define
\beno
(a) & \eqdefa & - \bigl (\partial_3 u_{\L}\cdot \nabla  w\big| \partial_3w\bigr)_{L^2}\,,\\
(b) & \eqdefa &- \bigl(  w\cdot \nabla \partial_3 u_{\L}\big | \partial_3w)_{L^2}\andf \\
(c) & \eqdefa &-\big( \partial_3 ( u_{\L}\cdot \nabla u_{\L}) \big | \partial_3 w\bigr)_{L^2}\,.
\eeno
The third term is the easiest.  By integration by parts and using the Cauchy-Schwarz inequality along with\refeq{estimfluctNL} we have
\beno
\Big | \int_0^\infty (c)(t) dt  \Big |& =&\Big | \int_0^\infty \int_{\R^3}  \partial_3^2\big (  \PP(u_{\L}\cdot \nabla u_{\L})(t,x) \big) \cdot   w(t,x) dxdt\Big | \\
 & \leq &\Bigl( \int_0^\infty   t^{\frac 32 } \bigl\|\partial_3^2  \PP(u_{\L}\cdot \nabla u_{\L}) (t)\bigr\|^2_{L^2} dt \Bigr)^{\frac 12}\Bigl( \int_0^\infty  \frac {\|w(t)\|^2_{L^2}} {t^{\frac 32}}  dt\Bigr)^{\frac 12}\\
& \leq &\sqrt{Q_{\L}^0  Q_{\L}^1 }\, \exp\Big(\frac12{\|u_0\|^2_{\dot B^{-1}_{\infty,2}} } \Big)  \,.
\eeno
Now let us estimate the contribution of~$(a)$ and~$(b)$. By integration by parts, we get, thanks to the divergence free condition on~$u_\L$,
$$
(a) = \bigl (\partial_3 u_{\L}\otimes w\big| \nabla  \partial_3w\bigr)_{L^2} \andf 
(b) = \bigl(  w\otimes \partial_3 u_{\L}\big |\nabla  \partial_3w)_{L^2}\,.
$$
The two terms can be estimated exactly in the same way since they are both of the form
 $$
 \int_{\R^3} w(t,x) \partial_3 u_{\L} (t,x) \nabla \partial_3 w(t,x) dx\, .
 $$
 We have
\beno
 \Big |\int_{\R^3} w(t,x) \partial_3 u_{\L} (t,x) \nabla \partial_3 w(t,x) dx \Big |&
\leq & \|w(t)\|_{L^2} \|\partial_3u_{\L}(t)\|_{L^\infty} \|\nabla
\partial_3 w\|_{L^2}\\
 & \leq &\frac 1 {100}  \|\nabla \partial_3 w\|_{L^2}^2 +100
\|w(t)\|^2_{L^2} \|\partial_3u_{\L}(t)\|_{L^\infty}^2\,.
 \eeno
The first term will be absorbed by the Laplacian. The second term can be
understood as a source term. By time integration, we get indeed
\beno
\int_0^T  \|w(t)\|^2_{L^2} \|\partial_3u_{\L}(t)\|_{L^\infty}^2 dt & \leq &
\int_0^T   \frac {\|w(t)\|^2_{L^2} } {t^{\frac 32} } \bigl(t^{\frac 34}
\|\partial_3u_{\L}(t)\|_{L^\infty}\bigr)^2 dt\\
& \leq &  \|\partial_3u_0\|^2_{\dot B^{-\frac 32}_{\infty,\infty} }
\int_0^\infty   \frac {\|w(t)\|^2_{L^2} } {t^{\frac 32} }dt\,,
\eeno
so it follows, thanks to Proposition~\ref{rescaledenergy}, that
$$
\begin{aligned}
 \int_0^T\int_{\R^3} w(t,x) \partial_3 u_{\L} (t,x) \nabla \partial_3 w(t,x) dxdt & \leq\frac 1 {100} \int_0^T  \|\nabla \partial_3 w(t)\|_{L^2}^2\, dt \\
 &\quad + C \|\partial_3u_0\|^2_{\dot B^{-\frac 32}_{\infty,\infty} }Q_{\L}^0 \, \exp  {\|u_0\|^2_{\dot B^{-1}_{\infty,2}} }   \,.
 \end{aligned}
 $$
The contribution of the quadratic term in~(\ref{eqpartial3w}) is estimated as follows: writing, for any function~$a$,
$$
\|a\|_{L^p_\h L^q_\v}\eqdefa \Big(\int \|a(x_1,x_2,\cdot)\|_{L^q(\R)}^p\, dx_1dx_2\Big)^\frac1p\,,
$$
we have by H\"older's inequality
\beno
\Big | \int_{\R^3}  \partial_3 w(t,x) \cdot \nabla w (t,x)  \partial_3 w(t,x) dx \Big |&
\leq & \|\partial_3 w(t)\|_{ L^2_\v L^4_\h}^2  \|\nabla
  w\|_{ L^\infty_\v L^2_\h} \\
& \leq &\|\partial_3 w(t)\| _{L^2} \|\nabla_\h \partial_3 w(t)\| _{L^2} \|\nabla  w(t)\| _{L^2} ^\frac12
 \|\nabla \partial_3 w(t)\| _{L^2}^\frac12 \,,
 \eeno
  where we have used the inequalities
\begin{equation}\label{anisotropy}
 \|a\|_{ L^\infty_\v L^2_\h} \lesssim \| \partial_3
a\|_{L^2}^{\frac12}\|a\|_{L^2}^{\frac12}\andf
  \|a \|_{L^2_\v L^4_\h}  \lesssim  \|a\|_{L^2}^{\frac12}
\|\nabla_\h a\|_{L^2}^{\frac12}
\end{equation}
with~$\nabla_\h\eqdefa(\partial_1,\partial_2)$.
The first inequality comes from
 \beno
 \|a(\cdot,x_3)\|_{L^2_\h}^2 & = & \frac 12 \int_{-\infty}^{x_3} \bigl(\partial_3 a(\cdot,z)\big | a(\cdot,z)\bigr)_{L^2_\h} dz\\
 & \leq & \frac12 \int_\R \|\partial_3 a(\cdot, z) \|_{L^2_\h} \|  a(\cdot, z) \|_{L^2_\h} dz\\
 & \leq & \|\partial_3 a\|_{L^2} \|  a\|_{L^2}
 \eeno
while the second simply comes from the embedding~$\dot H^\frac12_\h \subset L^4_\h$ and an interpolation.
By Young's inequality it follows that
$$
\begin{aligned}
\Big | \int_{\R^3}  \partial_3 w(t,x) \cdot \nabla w (t,x)  \partial_3 w(t,x) dx\Big |
&  \leq  \frac 1 {100}  \|\nabla \partial_3 w\|_{L^2}^2 + C\|\nabla  w(t)\| _{L^2} ^2\|\partial_3 w(t)\| _{L^2}^4\\
&  \leq  \frac 1 {100}  \|\nabla \partial_3 w(t)\|_{L^2}^2 \\
&\qquad\qquad
{}+ \displaystyle \Big( \sup_{t'\in[0,t]} \|\partial_3 w(t')\| _{L^2}^4\Big) t^\frac12   \frac{\|\nabla  w(t)\| _{L^2} ^2}{ t^\frac12 }\,\virgp
\end{aligned}
 $$
from which we infer by Proposition~\ref{rescaledenergy} that
$$
\begin{aligned}
\Big |  \int_{\R^3}  \partial_3 w(t,x) \cdot \nabla w (t,x)  \partial_3 w(t,x) dx\Big |& \leq\frac 1 {100}  \|\nabla \partial_3 w(t)\|_{L^2}^2 \\
&\quad +  \displaystyle\Big(\sup_{t'\in[0,t]} \|\partial_3 w(t')\| _{L^2}^4 \Big) t^\frac12 Q_{\L}^0\,\exp {\|u_0\|^2_{\dot B^{-1}_{\infty,2}} }  
\,.
 \end{aligned}
$$
 Finally   there holds after an integration by parts
\beno
 \int_{\R^3}  \partial_3 w(t,x) \cdot \nabla u_\L (t,x)  \partial_3 w(t,x) dx &
\leq & \|\partial_3 w(t)\|_{L^2} \|  u_\L (t)\|_{L^\infty}  \|\nabla
\partial_3 w(t)\|_{L^2 }  \\
&  \leq& \frac 1 {100}  \|\nabla \partial_3 w (t)\|_{L^2}^2 + C  \|\partial_3 w(t)\|_{L^2}^2\|  u_\L (t)\|_{L^\infty}^2 \,,
 \eeno
so plugging all these estimates together we infer  thanks to Gronwall's  inequality  that
$$
\begin{aligned}
\sup_{t\in [0,T] }  &\|\partial_3 w(t)\|_{L^2 }^2 + \int_0^T  \|\nabla \partial_3 w(t)\|_{L^2}^2\, dt\\  &\quad  \lesssim 
\Big(  T^\frac12 Q_{\L}^0 \displaystyle\sup_{t'\in[0,t]} \|\partial_3 w(t')\| _{L^2}^4 + \|\partial_3u_0\|^2_{\dot B^{-\frac 32}_{\infty,\infty} }Q_{\L}^0 + \sqrt{Q_{\L}^0  Q_{\L}^1 }\Big)  \exp\Big( 2{\|u_0\|^2_{\dot B^{-1}_{\infty,2}} } \Big)\,.
\end{aligned}
$$
Proposition~\ref{energypartial3} is proved.\qed

\subsection{End of the proof of Theorem~\ref{theolifespan}}\label{conclusionprooftheorem}
\subsubsection{Control of the fluctuation}
To make notation lighter let us set
$$
M_\L\eqdefa  \Big(\|\partial_3u_0\|^2_{\dot B^{-\frac 32}_{\infty,\infty} }Q_{\L}^0 \\   + \sqrt{Q_{\L}^0  Q_{\L}^1 }\Big)  \exp\big( 2{\|u_0\|^2_{\dot B^{-1}_{\infty,2}} } \big)\,.
$$
Proposition~\ref{energypartial3} provides the existence of a constant~$K$ such that the following a priori estimate holds
$$
\longformule{
\sup_{t\in [0,T] }  \|\partial_3 w(t)\|_{L^2 }^2 + \int_0^T  \|\nabla \partial_3 w(t)\|_{L^2}^2\, dt  
}
{ {}
 \leq K 
  T^\frac12 Q_{\L}^0\displaystyle\sup_{t\in[0,T]} \|\partial_3 w(t)\| _{L^2}^4   \exp\big( 2{\|u_0\|^2_{\dot B^{-1}_{\infty,2}} } \big)+ K M_\L \,.
}
$$
Let~$T^*$ be the maximal time of existence of~$u$, hence of~$w$, and recalling that~$w(t=0) = 0$, set~$T_1$ to be the maximal time~$T$ for   which
$$
 \sup_{t\in[0,T]} \|\partial_3 w(t)\| _{L^2}^2 \leq 2 K  M_\L\,.
 $$
 Then on~$[0,T_1]$ there holds
\beno
\sup_{t \in [0,T]}  \|\partial_3 w(t)\|_{L^2 }^2  + \int_0^T  \|\nabla \partial_3 w(t)\|_{L^2}^2\, dt   & \leq & 4K^3
  T_1^\frac12 Q_{\L}^0   M_\L^2 + K M_\L\\
 & \leq & KM_L\bigl(1+ 4K^2 T_1^\frac12 Q_{\L}^0   M_\L\bigr) \,.
\eeno
 This implies that 
 $$
 T_1 \geq T_* \quad \mbox{with} \quad 
  T_* \eqdefa  \left(\frac {1}{8K^2  Q_{\L}^0    M_\L} \right)^2\,\virgp
 $$
 and on~$[0,T_*]$ there holds
\begin{equation}\label{estimate0T*}
\sup_{t \in [0,T]}  \|\partial_3 w(t)\|_{L^2 }^2  + \int_0^T  \|\nabla \partial_3 w(t)\|_{L^2}^2\, dt  \leq \frac 32 K M_\L\,.
 \end{equation}

\subsubsection{End of the proof of the theorem}
Under the assumptions of Theorem~\ref{theolifespan} we know that there exists a unique solution~$u$ to (NS) on some time interval~$[0,T^*)$, which satisfies the energy estimate.    Let us prove that this time interval contains~$[0,T_*]$. Since the initial data~$u_0$ belongs to~$L^2$, we may assume  that~$u$ is a global Leray solution, meaning that
\begin{equation}\label{energyu}
\forall t \geq 0 \, , \quad \cE\big(u(t)\big) \leq \frac12\|u_0\|_{L^2}^2\,.
\end{equation}
Moreover one clearly has
$$
\sup_{t\geq 0}  \|\partial_3 u_\L(t)\|_{L^2 }^2+  \int_0^\infty \|\nabla\partial_3  u_\L (t)\|_{ L^2 }^2 \, dt\leq \|\partial_3 u_0\|_{ L^2 }^2
$$
so together with~(\ref{estimate0T*}) this implies that on~$[0,T_*]$,
\begin{equation}\label{controlgradd3u}
\sup_{t \in [0,T]}  \|\partial_3 w(t)\|_{L^2 }^2+\int_0^T \|\nabla\partial_3  u(t)\|_{ L^2 }^2 dt   \lesssim \|\partial_3 u_0\|_{ L^2 }^2 +    M_\L\,. 
\end{equation}
Let us prove that these estimates provide a control on~$u$  in~$\dot H^1$ on~$[0,T_*]$.
After differentiation of~(NS) with respect to the horizontal variables and
an energy estimate, we get for any~$\ell $ in~$ \{1,2\}$ and after an integration by parts
\beno
\frac12  \frac d {dt} \|\partial_\ell u(t)\|_{L^2}^2 +\|\nabla \partial_\ell
u(t)\|_{L^2}^2 & = & - \int_{\R^3} \partial_\ell (u\cdot \nabla u )\cdot
\partial_\ell u  \,(t,x)\, dx\\
& \leq &\|u\|_{L^\infty_\v L^4_\h} \|\nabla u(t)\|_{L^2_\v L^4_\h} \|\partial_\ell^2
u(t)\|_{L^2} \, .
\eeno
Similarly to~(\ref{anisotropy})  
we have
 \beno
 \|u\|_{L^\infty_\v L^4_\h}^2 & \lesssim &
  \|u\|_{L^\infty_\v \dot H^\frac12_\h}^2 \\
  & \lesssim & \int_{-\infty}^{x_3} \bigl(\partial_3 u(\cdot,z)\big | u(\cdot,z)\bigr)_{\dot H^\frac12_\h} dz\\
 & \lesssim & \|\partial_3 u\|_{L^2} \| \nabla_\h u\|_{L^2}
 \eeno
so using~(\ref{anisotropy})  
we infer that
\beno
 \Big|   \int_{\R^3} \partial_\ell (u\cdot \nabla u )\cdot
\partial_\ell u  \,(t,x)\, dx   \Big|& \leq
 & C
 \|\partial_3 u(t)\|_{L^2} ^\frac12 \| \nabla_\h u(t)\|_{L^2} ^\frac12
  \|\nabla u(t)\|_{L^2}  ^\frac12 
   \|\nabla \nabla_\h u(t)\|_{L^2}  ^\frac12 
 \|\partial_\ell^2 u(t)\|_{L^2}\\
& \leq & \frac 1 {100}  \|\nabla \nabla_\h  u(t)\|_{L^2}^2 +C
 \|\partial_3 u\|_{L^2} ^2 \| \nabla_\h u\|_{L^2} ^2
  \|\nabla u(t)\|_{L^2}^2 \, .
\eeno
We obtain
$$
 \frac d {dt} \|\nabla_\h u(t)\|_{L^2}^2 +\|\nabla \nabla_\h 
u(t)\|_{L^2}^2 \lesssim  \|\partial_3 u\|_{L^2} ^2 \| \nabla_\h u\|_{L^2} ^2
  \|\nabla u(t)\|_{L^2}^2 \, ,
$$
and Gronwall's inequality implies that
$$
\|\nabla_\h  u(t)\|_{L^2}^2 +\int_0^t\|\nabla \nabla_\h 
u(t')\|_{L^2}^2 dt'
\leq \|\nabla_\h  u_0\|_{L^2}^2 \exp \Bigl( \int_0^t \|\partial_3 u(t')\|_{L^2} ^2
  \|\nabla u(t')\|_{L^2}^2 dt'\Bigr) \, .
$$
The fact that we control~$\|\nabla u\|_{L^2_t(L^2_x)}$  and~$\| \partial_3  u\|_{L^\infty_t(L^2_x)}$ thanks to~(\ref{energyu}) and~(\ref{controlgradd3u}) implies that on~$[0,T_*]$ there holds
$$
\sup_{t \in [0,T]}    \|\nabla u(t)\|_{ L^2  }^2 + \int_0^T \|\nabla^2 u(t)\|_{L^2 } ^2 \, dt \leq  \|\nabla u_0\|_{ L^2  }^2\exp \Bigl( \|  u_0\|_{ L^2  } (KM_\L)^\frac12  \Bigr) \, .$$
This means that there is a unique, smooth solution at least on~$[0,T_*]$, and   Theorem~\ref{theolifespan} is proved. \qed

\section{Comparison of both life spans: proof of Theorem \ref{theoexample}}
\label{example}
Let us introduce the notation
$$
 f_\e(x_1,x_2,x_3)\eqdefa  \cos \Bigl( \frac {x_1} \e\Bigr) f\Bigl( x_1,\frac{x_2}{\e^\al}, x_3\Bigr)\,,
$$
  where~$\e$ is a given number, assumed to be small, and~$\alpha  $ is a fixed parameter in the open interval~$]0,1[$. We assume the initial data is given by the following expression
\begin{equation}
\label{definitialdata}
u_{0,\e} (x) = \frac{A_ \e}\e  \bigl( 0,\e^\al(-\partial_3\f)_\e, (\partial_2\f)_\e\bigr)
\end{equation}
where~$\f$ is a smooth compactly supported function and the parameter~$A_\e \gg 1$ will be tuned later.
  
\medbreak
Let us recall that Lemma~3.1 of~\cite{cg2} claims in particular that 
\beq
\label {recallLemma31cg2}
\forall \s>0\,,\ \|f_{\varepsilon}\|_{\dot B^{-\sigma}_{p,1}} \leq C_\s \e^{\sigma + \frac\alpha p}
\andf \|f_{\e}\|_{\dot B^{-\s}_{\infty,\infty}}\geq c_\s \e^{\s}.
\eeq
This implies   that 
\beq
\label {recallLemma31cg2bis}
\|u_{0,\e}\|_{\dot B^{-1+2\g}_{\infty,\infty}} \lesssim A_\e \e^{-2\g}\,,\  \|u_{0,\e} \|_{\dot B^{-1}_{\infty,\infty}}\sim\|u_{0,\e} \|_{\dot B^{-1}_{\infty,2}}\sim  A_\e \quad \mbox{and }\quad \|\partial_3 u_{0,\e} \|_{\dot B^{-\frac32}_{\infty, \infty}}\lesssim A_\e \e^\frac12\,.
\eeq
With the notation  of Theorem \refer  {wpgnsBesov} there holds therefore.
$$
T_{\rm FP} (u_{0,\e})\geq C\e^{2} A_\e^{-\frac1\gamma}\,.
$$
Let us now compute~$T_\L(u_{0,\e})$. Recalling that~$u_{\L}(t) = e^{t\Delta}u_{0,\e}$, we can write 
\beno
u_L^{1} \partial_{1} u_L^{1} + u_L^{2} \partial_{2} u_L^{1}  
& = &   \Bigl(\frac {A_ \e}\e\Bigr)^2  e^{t\D}f_{\e} e^{t\D}g_{\e} \andf\\
u_L^{1} \partial_{1} u_L^{2} + u_L^{2} \partial_{2} u_L^{2}
& = &   \Bigl(\frac {A_ \e}\e\Bigr)^2 e^{t\D}\wt f_{\e}
e^{t\D}\wt g_{\e}.
\eeno
where~$f$, $g$, $\wt f$, $\wt g$  are  smooth compactly supported functions.
Now let us estimate
$$
\int_0^\infty t^{\frac 12} \bigl \| e^{t\D} f_\e \,  e^{t\D} g_\e\bigr\|^2_{L^2}  dt\,.
$$
for~$f$ and~$g$ given smooth compactly supported functions.
We write
\beno
\int_0^\infty t^{\frac 12} \bigl \| e^{t\D} f_\e \,  e^{t\D} g_\e\bigr\|^2_{L^2}  dt & = & 
\int_0^\infty t^{\frac 32} \bigl \| e^{t\D} f_\e \,  e^{t\D} g_\e\bigr\|^2_{L^2}  \frac {dt } t\\
&\leq & \int_0^\infty \bigl(t^{\frac 38} \| e^{t\D} f_\e\|_{L^4}\bigr)^2\bigl( t^{\frac 38} \|e^{t\D} g_\e\|\bigr)^2_{L^4}  \frac {dt } t
\eeno
thanks to the H\"older inequality.
The Cauchy-Schwarz inequality and the definition of Besov norms imply that 
\beno
\int_0^\infty t^{\frac 12} \bigl \| e^{t\D} f_\e e^{t\D} g_\e\bigr\|^2_{L^2}  dt 
&\leq & \Bigl(\int_0^\infty \bigl(t^{\frac 38} \| e^{t\D} f_\e\|_{L^4}\bigr)^4 \frac {dt } t\Bigr)^{\frac 12}
\Bigl(\int_0^\infty \bigl(t^{\frac 38} \| e^{t\D} g_\e\|_{L^4}\bigr)^4 \frac {dt } t\Bigr)^{\frac 12}\\
& \leq & \|f_\e\|_{\dot B^{-\frac 34}_{4,4}}^2 \|g_\e\|_{\dot B^{-\frac 34}_{4,4}}^2\, .
\eeno
It is easy to check that
$$
 \|f_\e\|_{\dot B^{-\frac 34}_{4,4}} \lesssim \e^{\frac {3+\al} 4}\,,
$$
so it follows (since~$\PP$ is a homogeneous Fourier multiplier of order~0) that
\beq
\label  {eqtimeQ0L}
Q_{\L}^0 \lesssim A_\e^4 \e^{\alpha-1}\,.
\eeq
For the initial data~(\ref{definitialdata}), differentations with respect to the vertical variable~$\partial_3$ have no real influence  on the term~Ê$u_{\L}(t)\cdot\nabla u_{\L}(t)$. Indeed, we have
$$
\partial_3^2 \bigl( u_{\L}(t)\cdot\nabla u_{\L}(t) \bigr) = \partial_3^2  u_{\L}(t)\cdot\nabla u_{\L}(t) +
2\partial_3  u_{\L}(t)\cdot  \partial_3\nabla u_{\L}(t)+  u_{\L}(t)\cdot\partial_3^2\nabla u_{\L}(t)
$$
and  it is then obvious that~$\partial_3^2 \bigl( u_{\L}(t)\cdot\nabla u_{\L}(t) \bigr)$ is a sum of term of the type
$$
 \Bigl(\frac {A_ \e}\e\Bigr)^2  \,e^{t\D} f_\e  \, e^{t\D} g_\e\,.
$$
Then following the lines used to   estimate the term~$Q_L^0$, we write
\beno
\int_0^\infty t^{\frac 3 2} \bigl \| e^{t\D} f_\e e^{t\D} g_\e\bigr\|^2_{L^2}  dt 
&\leq & \Bigl(\int_0^\infty \bigl(t^{\frac 58} \| e^{t\D} f_\e\|_{L^4}\bigr)^4 \frac {dt } t\Bigr)^{\frac 12}
\Bigl(\int_0^\infty \bigl(t^{\frac 58} \| e^{t\D} g_\e\|_{L^4}\bigr)^4 \frac {dt } t\Bigr)^{\frac 12}\\
& \leq & \|f_\e\|_{\dot B^{-\frac 54}_{4,4}}^2 \|g_\e\|_{\dot B^{-\frac 54}_{4,4}}^2\, .
\eeno
 It is easy to check that
$$
 \|f_\e\|_{\dot B^{-\frac 54}_{4,4}} \lesssim \e^{\frac {5+\al} 4}\,,
$$
so it follows that
$$
Q_{\L}^1 \lesssim A_\e^4\e^{\alpha+1}\,.
$$
Together with\refeq {recallLemma31cg2bis} and\refeq {eqtimeQ0L}, we infer   that 
\beno
Q_{\L}^0 \Big(
\|\partial_3u_0\|_{\dot B^{-\frac32}_{\infty,\infty}}^2 Q_{\L}^0+ \sqrt{Q_{\L}^0Q_{\L}^1}
\Big) \exp \big(4
\|u_0\|^ 2_{\dot B^{-1}_{\infty,2}} \bigr) 
& \lesssim & 
A_\e^4 \e^{\alpha-1}  \bigl( A_\e^6\e^{\al}+A_\e^4\e^\al\bigr)\exp (C_0 A_\e^2) \\
& \lesssim & A_\e^{10} \e^{2\al-1} \exp (C_0 A_\e^2)
\eeno
because~$A_\e$ is larger than~$1$.
Let us choose some~Ê$\kappa$ in~$]0,\eta[$ and then
$$
A_\e  \eqdefa \Bigl( \frac  {C_0} {-\kappa \log  \e}\Bigr)^{\frac 12}\cdotp
$$Then with the notation of Theorem~\ref{theolifespan} we have 
$$
T_\L = CA_\e^{-20} \e^{2(1-2\alpha+\kappa)}  \,.
$$
Let us choose~$\kappa'$ in~$]\kappa,\eta[$. By definition of~$A_\e$ we get  that
$$
T_L\geq C \e^{2(1-2\alpha+\kappa')}
$$
Choosing $\ds
\al = 1-\frac {\eta-\kappa'} 4$
concludes the proof of Theorem \ref{theoexample}. \qed
 
 \appendix
 \section{A Littlewood-Paley toolbox}
 Let us recall some well-known results on Littlewood-Paley theory (see for instance\ccite{bcd} for more details).
 \begin{defin}
{\sl Let $\phi \in \mathcal{S}(\mathbb{R}^{3})$ be such that
$\widehat\phi(\xi) = 1$ for $|\xi|\leq 1$ and $\widehat\phi(\xi)= 0$ for
$|\xi|>2$. We define, for~$j \in \Z$, the function~$\phi_{j}(x)\eqdefa
2^{3j}\phi(2^{j}x)$, and the Littlewood--Paley operators
$$ S_{j} \eqdefa \phi_{j}\ast\cdot \quad \mbox{and} \quad \Delta_{j}
\eqdefa S_{j+1} - S_{j}\,.
$$ 
}
\end{defin}
 Homogeneous Sobolev spaces are defined by the norm
$$
\|a\|_{\dot H^s}\eqdefa \Big(\sum_{j \in \ZZ} 2^{2js} \|\Delta_j a\|_{L^2}^2\Big)^\frac12\,.
$$ 
This norm is equivalent to
$$
\|a\|_{\dot H^s}\sim  \Big(\int_{\R^3} |\xi|^{2s}
|\cF a(\xi)|^2 \, d\xi
\Big)^\frac12\,,
$$
where~$\cF$ is the Fourier transform.
Finally let us  recall the definition of Besov norms of negative index.
 \begin{defin}
{\sl 
Let~$\s$  be a positive real number and~$(p,q)$ in~$[1,\infty]^2$. Let us define the homogeneous Besov norm~$\|\cdot\|_{\dot B^{-\s}_{p,q}}$ by
$$
 \|a\|_{\dot B^{-\s}_{p,q}} = \bigl\| t^{\frac \s 2}\| e^{t\D} a\|_{L^p}\bigr\|_{L^q\left(\R^+;\frac {dt} t\right)}\,.
$$
}
\end{defin}
Let us mention that thanks to  the properties of the heat flow, for~$p_1\leq p_2$ and~Ê$q_1\leq q_2$, we have the following inequality, valid for any regular function~$a$
$$
 \|a\|_{\dot B^{-\s-3\left (\frac 1 {p_1} -\frac 1 {p_2}\right)}_{p_2,q}} \lesssim \|a\|_{\dot B^{-\s}_{p_1,q}} \andf
 \|a\|_{\dot B^{-\s}_{p,q_2}} \lesssim \|a\|_{\dot B^{-\s}_{p,q_1}}.
$$
An equivalent definition using the Littlewood-Paley decomposition is
$$
 \|a\|_{\dot B^{-\s}_{p,q}} \sim 
 \Big(\sum_{j \in \ZZ} 2^{-j\sigma q} 
 \|\Delta_j a\|_{L^p}^q\Big)^\frac1q\,.
$$

 \end{document}